\documentclass[11pt]{amsart}
\usepackage[margin=1in]{geometry}

\usepackage{amssymb}
\usepackage{amsthm}
\usepackage{amsmath}
\usepackage{mathrsfs}
\usepackage{amsbsy}
\usepackage[all]{xy}
\usepackage{bm}
\usepackage{hyperref}
\usepackage{tikz}
\usepackage{array}
\usepackage{float}
\usepackage{enumerate}
\usepackage{xcolor}
\usepackage{hhline}
\allowdisplaybreaks
\usepackage[noadjust]{cite}

\usepackage{caption}
\usepackage{subcaption}
\usepackage{tabu}
\usepackage{diagbox}
\usepackage{tikz}
\usepackage{bbm}
\usepackage{booktabs}

\DeclareFontFamily{U}{rcjhbltx}{}
\DeclareFontShape{U}{rcjhbltx}{m}{n}{<->rcjhbltx}{}
\DeclareSymbolFont{hebrewletters}{U}{rcjhbltx}{m}{n}

\let\aleph\relax\let\beth\relax

\DeclareMathSymbol{\aleph}{\mathord}{hebrewletters}{39}
\DeclareMathSymbol{\beth}{\mathord}{hebrewletters}{98}

\usepackage[noabbrev,capitalise]{cleveref}

\newenvironment{enumerate*}%
  {\begin{enumerate}[(I)]%
    \setlength{\itemsep}{10pt}%
    \setlength{\parskip}{0pt}}%
  {\end{enumerate}}

\newtheorem{theorem}{Theorem}[section]
\newtheorem{proposition}[theorem]{Proposition}

\newtheorem{lemma}[theorem]{Lemma}

\theoremstyle{definition}

\title{Relative sizes of iterated sumsets}

\author[]{Noah Kravitz}
\address[]{Department of Mathematics, Princeton University, Princeton, NJ 08540, USA}
\email{nkravitz@princeton.edu}

\begin{document}

\maketitle

\begin{abstract}
Let $hA$ denote the $h$-fold sumset of a subset $A$ of an abelian group.  Resolving a problem of Nathanson, we show that for any prescribed permutations $\sigma_1, \ldots, \sigma_H \in \mathfrak{S}_n$, there exist finite subsets $A_1, \ldots, A_n \subseteq \mathbb{Z}$ such that for each $1 \leq h \leq H$, the relative order of the quantities $|h A_1|, \ldots, |h A_n|$ is given by $\sigma_h$.  We also establish extensions where $\mathbb{Z}$ is replaced by any other infinite abelian group or where one prescribes some equalities (not only inequalities) among the sumset sizes.
\end{abstract}

\section{Introduction}
For a natural number $h$ and a subset  $A$ of an abelian group, let
$$hA:=\{a_1+\cdots+a_h: a_1, \ldots, a_h \in A\}$$
denote the \emph{$h$-fold sumset} of $A$.  The quantitative growth of the sequence $|A|,|2A|,|3A|, \ldots$ is controlled by tools such as the Pl\"unnecke--Ruzsa Inequality (see, e.g., \cite{petridis}).  

Nathanson~\cite{nathanson} recently posed a suite of more qualitative questions about the possible \emph{relative} growth rates of such sequences for different choices of $A \subseteq \mathbb{Z}$.  For subsets $A_1, \ldots, A_n \subseteq \mathbb{Z}$ and a natural number $h$, one can consider the relative order of the quantities
$$|hA_1|, |hA_2|, \ldots, |hA_n|.$$
If these quantities are all distinct, then there is a unique permutation $\sigma \in \mathfrak{S}_n$ which (when written in $1$-line notation) has the same relative order as $|hA_1|, |hA_2|, \ldots, |hA_n|$.

Nathanson asked if for prescribed permutations $\sigma_1, \ldots, \sigma_R \in \mathfrak{S}_n$, one can always find an increasing sequence $h_1<\cdots<h_R$ of natural numbers and finite subsets $A_1, \ldots, A_n \subseteq \mathbb{Z}$ such that $|h_rA_1|, \ldots, |h_r A_n|$ has the same relative order as $\sigma_r$ for each $1 \leq r \leq R$.  Nathanson further asked if one can prescribe the sequence $h_1<\cdots<h_R$ in addition to the permutations $\sigma_1, \ldots, \sigma_R$; note that an affirmative answer for $(h_1, h_2, \ldots, h_R)=(1,2,\ldots, R)$ (with $R$ arbitrary) would imply an affirmative answer in general.  Our main result establishes precisely this fact, not only in the integers but in any sufficiently large abelian group.

\begin{theorem}\label{thm:main}
Let $n,H \in \mathbb{N}$.  Then for every sufficiently large abelian group $G$ and any permutations $\sigma_1, \ldots, \sigma_H \in \mathfrak{S}_n$, there exist finite subsets $A_1, \ldots, A_n \subseteq G$ such that
$$|hA_1|, \ldots, |hA_n| \quad \text{has the same relative order as $\sigma_h$}$$
for each $1 \leq h \leq H$.
\end{theorem}

This theorem is optimal in the sense that one cannot hope to control the relative sizes of the quantities $|hA_k|$ for infinitely many values of $h$, even along a sparse sequence.  Indeed, a result of Khovanskii~\cite{khov1,khov2} (see also~\cite{nathanson3,nathanson2,NR}) shows that for any finite subset $A \subseteq G$, the quantity $|hA|$ is eventually a polynomial function of $h$; hence the relative order of $|h A_1|, \ldots, |hA_n|$ is the same for all sufficiently large $h$.

We also remark that the conclusion of \Cref{thm:main} fails unless the group $G$ is sufficiently large in terms of both $n$ and $H$. For instance, if $2^{|G|}<n$, then the $A_k$'s cannot all be distinct.  More subtly, for any $A \subseteq G$, the quantity $|hA|$ is constant for all $h \geq |G|$, so the relative order of $|hA_1|, \ldots, |hA_n|$ is the same for all $h \geq |G|$; this shows that the conclusion of \Cref{thm:main} fails for $|G|<H$.

One could hope to strengthen \Cref{thm:main} by prescribing more conditions.  First, one could ask about prescribing equalities (in addition to inequalities) among the iterated sumset sizes.  Second, recall that the relative order of $h$-fold iterated sumsets is eventually constant for sufficiently large $h$.  One could ask about prescribing this ``limiting'' relative order in addition to the relative orders for the first several iterated sumsets.  In the integer setting, we can prove an extension of \Cref{thm:main} that makes both of these improvements.

\begin{theorem}\label{thm:extension}
Let $n,H \in \mathbb{N}$.  Then for any tuples $\tau_1, \ldots, \tau_H, \tau_\infty \in \mathbb{N}^n$, there exist finite subsets $A_1, \ldots, A_n \subseteq G$ such that
$$|hA_1|, \ldots, |hA_n| \quad \text{has the same relative order as $\tau_h$}$$
for each $1 \leq h \leq H$ and $$|hA_1|, \ldots, |hA_n| \quad \text{has the same relative order as $\tau_\infty$}$$
for all $h>H$.
\end{theorem}
We leave it as an open problem to establish the analogous result in all infinite abelian groups.

\subsection{Proof strategy and paper outline}
The construction for \Cref{thm:main} is based on combinations of basic ``building blocks''.  For each $h$ we construct a family of $n$ building blocks whose $h$-fold iterated sumsets have distinct sizes but whose $h'$-fold iterated sumsets have the same size for each $h'>h$.  We then construct the sets $A_k$ as suitable Cartesian products of the building blocks.

We carry out this strategy in \Cref{sec:proofs}: We show that it suffices to prove \Cref{thm:main} for the two ``model cases'' $G=\mathbb{Z}$ and $G=(\mathbb{Z}/p\mathbb{Z})^N$, then we give a precise description of the necessary properties of our building blocks, and finally we show how to construct the building blocks in each model case.  In \Cref{sec:extension} we adapt the proof of \Cref{thm:main} in order to prove \Cref{thm:extension}.  In \Cref{sec:old-construction} we describe an alternative construction for a weaker version of \Cref{thm:main} in the integer setting; the proof introduces multiscale arguments that may be of independent interest.

\section{Proof of the main theorem}\label{sec:proofs}

\subsection{Reduction to model cases}\label{sec:reduction}
We begin by reducing \Cref{thm:main} to the following two propositions, whose proofs occupy the remainder of this section.

\begin{proposition}\label{prop:positive}
Let $p$ be a prime, and let $n,H \in \mathbb{N}$.  Then there is some $N_p=N_p(n,H) \in \mathbb{N}$ such that for any permutations $\sigma_1, \ldots, \sigma_H \in \mathfrak{S}_n$, there exist subsets $A_1, \ldots, A_n \subseteq (\mathbb{Z}/p\mathbb{Z})^{N_p}$ such that
$$|hA_1|, \ldots, |hA_n| \quad \text{has the same relative order as $\sigma_h$}$$
for each $1 \leq h \leq H$.
\end{proposition}

For integers $M \leq N$, write $[M,N]:=\{M, M+1,\ldots, N\}$.  For $N \in \mathbb{N}$, write $[N]:=[0,N]$.

\begin{proposition}\label{prop:zero}
Let $n,H \in \mathbb{N}$.  Then there is some $N_\infty=N_\infty(n,H) \in \mathbb{N}$ such that for any permutations $\sigma_1, \ldots, \sigma_H \in \mathfrak{S}_n$, there exist subsets $A_1, \ldots, A_n \subseteq [N_\infty] \subseteq \mathbb{Z}$ such that
$$|hA_1|, \ldots, |hA_n| \quad \text{has the same relative order as $\sigma_h$}$$
for each $1 \leq h \leq H$.
\end{proposition}

\begin{proof}[Proof of \Cref{thm:main}, assuming Propositions~\ref{prop:positive} and~\ref{prop:zero}]
Fix $n,H \in \mathbb{N}$ and $\sigma_1, \ldots, \sigma_H \in \mathfrak{S}_n$, as in the statement of \Cref{thm:main}.  Take $N_p=N_p(n,H), N_\infty=N_\infty(n,H)$ as in the previous two propositions.  Let $p_1, \ldots, p_s$ be the primes up to $HN_\infty$.

Suppose that $G$ is an abelian group of size at least
$$N:=(HN_\infty)^{\sum_{i=1}^s N_{p_i}}.$$
If $G$ has an element $x$ of order larger than $HN_\infty$, then \Cref{prop:zero} lets us find the desired sets $A_1, \ldots, A_n$ in the set $\{0,x,2x, \ldots, N_\infty x\}$; note that $x$ has large enough order to prevent wrap-around when we take sumsets.  It remains to consider the case where every element of $G$ has order at most $HN_\infty$.  By passing to subgroups if necessary, we may assume that $G$ is finite.  Hence we can write
$$G=\prod_{i=1}^s G_{i},$$
where each $G_i$ is a $p_i$-group (i.e., every element of $G_i$ has order a power of $p_i$).  The Pigeonhole Principle provides some $1 \leq i \leq s$ such that
$$|G_i| \geq (HN_\infty)^{N_{p_i}}$$
The Fundamental Theorem of Finitely Generated Abelian Groups lets us write $G_i$ as a product of cyclic groups $\mathbb{Z}/p_i^j\mathbb{Z}$ with $p_i \leq p_i^j \leq HN_\infty$ (the upper bound due to our assumption that $G$ does not have any elements of order larger than $HN_\infty$).  Hence there are at least $N_{p_i}$ multiplicands in the product, so the $p_i$-torsion subgroup of $G_i$ is isomorphic to $(\mathbb{Z}/p_i\mathbb{Z})^{N'}$ for some $N' \geq N_{p_i}$.  \Cref{prop:positive} now lets us find the desired sets $A_1, \ldots, A_n$ in this subgroup.
\end{proof}

One can extract explicit values of $N_p,N_\infty$ from our proofs of Propositions~\ref{prop:positive} and~\ref{prop:zero}, so the value of $N$ in the proof of \Cref{thm:main} can also be explicitly computed.

\subsection{The general strategy}
In order to highlight the common structure of the proofs of Propositions~\ref{prop:positive} and~\ref{prop:zero}, we describe the general framework here before diving into the details.  The main idea is constructing the sets $A_k$ as $H$-fold Cartesian products, where for each $1 \leq h \leq H$ a different  multiplicand ``dominates'' the sizes of the quantities $|hA_k|$.  One should think of the sets $B_{h,i}$ in the following proposition as the building blocks that will dominate $h$-fold sumsets of our sets $A_k$.

\begin{proposition}\label{prop:putting-together-the-pieces}
Let $n,H \in \mathbb{N}$, and let $G$ be an abelian group.  Suppose that for each $1 \leq h \leq H$ there are $t_h \in \mathbb{N}$ and finite subsets $B_{h,1}, \ldots, B_{h,n} \subseteq G^{t_h}$ such that
\begin{equation}\label{eq:first-condition}
|hB_{h,1}|<|hB_{h,2}|<\cdots<|hB_{h,n}|
\end{equation}
and
\begin{equation}\label{eq:second-condition}
|h'B_{h,1}|=|h'B_{h,2}|=\cdots=|h'B_{h,n}| \quad \text{for all $h'>h$}.
\end{equation}
Then for any permutations $\sigma_1, \ldots, \sigma_H \in \mathfrak{S}_n$, there exist $t \in \mathbb{N}$ and finite subsets  $A_1, \ldots, A_n \subseteq G^t$ such that
$$|hA_1|, \ldots, |hA_n| \quad \text{has the same relative order as $\sigma_h$}$$
for each $1 \leq h \leq H$.
\end{proposition}

\begin{proof}
Let $\mu_1<\cdots<\mu_H$ be a quickly-increasing sequence of natural numbers, to be determined later.  Set $$t:=\mu_1t_1+\cdots+\mu_Ht_H.$$  For each $1 \leq k \leq n$, let $A_k \subseteq G^t$ be the Cartesian product of $\mu_1$ copies of $B_{1,\sigma_1(k)}$, and $\mu_2$ copies of $B_{2,\sigma_2(k)}$, and so on, up to $\mu_H$ copies of $B_{H,\sigma_H(k)}$.  Let us check that iterated sumsets of the sets $A_k$ have the desired relative orders.  Fix $1 \leq h \leq H$.  For each $k$, we have
$$|hA_k|=\prod_{j=1}^H |hB_{j,\sigma_j(k)}|^{\mu_j}= \underbrace{\left(\prod_{j=1}^{h-1} |hB_{j,\sigma_j(k)}|^{\mu_j} \right)}_{(j<h)} \cdot \underbrace{\left( |hB_{h,\sigma_h(k)}|^{\mu_h}\right)}_{(j=h)} \cdot \underbrace{\left(\prod_{j=h+1}^H |hB_{j,\sigma_j(k)}|^{\mu_j} \right)}_{(j>h)}.$$
By assumption, the $j>h$ contribution is the same for all $1 \leq k \leq n$, so we can ignore it.  The $j=h$ contributions are in the desired relative order.  If $\mu_h$ is chosen sufficiently large relative to $\mu_1, \ldots, \mu_{h-1}$, then this contribution will dominate the $j<h$ contribution, and the quantities $|hA_k|$ will have the desired relative order.
\end{proof}

We remark that the $n=2$ hypothesis of \Cref{prop:putting-together-the-pieces} is enough to imply the hypothesis for all $n$: If we have sets $B_{h_1}, B_{h,2}$ satisfying \eqref{eq:first-condition} and \eqref{eq:second-condition}, then the sets $B'_{h,i}:=B_{h,1}^{\times i} \times B_{h,2}^{\times n-i}$ (for $1 \leq i \leq n$) also satisfy \eqref{eq:first-condition} and \eqref{eq:second-condition}.

Now \Cref{prop:positive} (the main theorem for $(\mathbb{Z}/p\mathbb{Z})^N$) is an immediate consequence of \Cref{prop:putting-together-the-pieces} (with $G=\mathbb{Z}/p\mathbb{Z}$) once we find suitable sequences of sets satisfying \eqref{eq:first-condition} and \eqref{eq:second-condition}.  The deduction of \Cref{prop:zero} (the main theorem for $\mathbb{Z}$) is only slightly more involved: \Cref{prop:putting-together-the-pieces} with $G=\mathbb{Z}$ produces subsets of $\mathbb{Z}^t$, and then one can transfer these sets to $\mathbb{Z}$ by applying a suitable Freiman homomorphism of order $H$.

It remains to find sequences of subsets of $G=\mathbb{Z}/p\mathbb{Z}$ and of $G=\mathbb{Z}$ satisfying \eqref{eq:first-condition} and \eqref{eq:second-condition}.

\subsection{Positive characteristic}\label{sec:positive}

Let $p$ be a prime.  Fix $n,H \in \mathbb{N}$, and let $1 \leq h \leq H$.  Our goal is to find $t \in \mathbb{N}$ and finite subsets $B_1, \ldots, B_n \subseteq (\mathbb{Z}/p\mathbb{Z})^t$ satisfying \eqref{eq:first-condition} and \eqref{eq:second-condition}.  (For notational simplicity, we write $t,B_{i}$ instead of $t_h,B_{h,i}$.)

For integers $0 \leq s \leq t$, let $X(s,t)$ denote the subset of $(\mathbb{Z}/p\mathbb{Z})^t$ consisting of all elements with at most $s$ nonzero coordinates.  The key property of the sets $X(s,t)$ is that for any $j \in \mathbb{N}$, we have
$$jX(s,t)=X(\min\{js,t\},t).$$
We record two consequences:
\begin{enumerate}[(i)]
    \item \label{conseq-1}  We have the strict inequality $$|jX(s,t)|=|X(js,t)|<|X(js',t)|=|jX(s',t)|$$ for all $0\leq s<s' \leq t/j$.
    \item \label{conseq-2} We have the identity $|jX(s,t)|=p^t$ for all $j \geq t/s$.
\end{enumerate}

The construction of the sets $B_i$ is now quite simple.  Take natural numbers $t, s_1, \ldots, s_n$ satisfying
$$t/(h+1) \leq s_1<s_2<\cdots<s_n \leq t/h$$
(for instance, one could take $t:=h(h+1)(n-1)$ and $s_i:=h(n-1)+i-1$), and set
$$B_i:=X(s_i,t)$$
for $1 \leq i \leq n$.  Since $hs_n\leq t$, consequence (\ref{conseq-1}) from above gives the string of inequalities
$$|hB_1|<\cdots <|hB_n|;$$
thus the sets $B_i$ satisfy condition \eqref{eq:first-condition}.  Since $(h+1)s_1\geq t$, consequence (\ref{conseq-2}) from above gives that $|h'B_i|=p^t$
is constant for all $1 \leq i \leq n$ and $h'>h$; thus the sets $B_i$ also satisfy condition \eqref{eq:second-condition}.

Applying \Cref{prop:putting-together-the-pieces} with these sets $B_i$ proves \Cref{prop:positive}.

\subsection{Characteristic zero}\label{sec:zero}
Fix $n,H \in \mathbb{N}$, and let $1 \leq h \leq H$.  Our goal is to find finite subsets $B_1, \ldots, B_n \subseteq \mathbb{Z}$ satisfying \eqref{eq:first-condition} and \eqref{eq:second-condition}.  (We again write $B_{i}$ instead of $B_{h,i}$ for brevity.)  Recall that we write $[M,N]:=\{M, M+1,\ldots, N\}$ for integers $M \leq N$.

For integers $0 \leq u \leq v$, let $Y(u,v):=[0,u] \cup [v-u,v] \subseteq \mathbb{Z}$.  For any $j \in \mathbb{N}$, we have
$$jY(u,v)=\bigcup_{\ell=0}^j ((j-\ell) [0,u]+\ell[v-u,v])=\bigcup_{\ell=0}^j [\ell(v-u), \ell(v-u)+j u].$$
We record two consequences:
\begin{enumerate}[(i)]
\setcounter{enumi}{2}
    \item \label{conseq-3} For any $0 \leq u<u'\leq v$, we have the containment $Y(u,v) \subseteq Y(u',v)$, and hence $$jY(u,v) \subseteq jY(u',v)$$ for all $j \in \mathbb{N}$.  If moreover $u'\leq v/(j+1)$, then this last containment is strict because the element $ju'$ is in $jY(u',v)$ but not in $jY(u,v)$.
    \item \label{conseq-4} We have the identity $jY(u,v)=[0,jv]$ for all $j \geq (v-1)/u-1$.
\end{enumerate}

The construction of the sets $B_i$ is again merely a matter of choosing the parameters $u,v$ appropriately.  Take natural numbers $v, u_1, \ldots, u_n$ satisfying
$$(v-1)/(h+2) \leq u_1<u_2<\cdots<u_n \leq v/(h+1)$$
(with much flexibility, as in the previous subsection), and set
$$B_i:=Y(u_i,v)$$
for $1 \leq i \leq n$.  Now consequences (\ref{conseq-3}) and (\ref{conseq-4}) ensure that the sets $B_i$ satisfy conditions \eqref{eq:first-condition} and \eqref{eq:second-condition}, and an application of \Cref{prop:putting-together-the-pieces} proves \Cref{prop:zero}.

\section{Prescribing more conditions}\label{sec:extension}

In this short section we prove \Cref{thm:extension}, which ``upgrades'' \Cref{thm:main} in the integers.  Recall that the two improvements in \Cref{thm:extension} are the ability to prescribe equalities (in addition to inequalities) among the first few iterated sumset sizes and the ability to dictate the limiting relative order of the iterated sumset sizes.  We obtain the first improvement by replacing the building block $Y(u,v)$ with a slightly more complicated set.  A similar trick leads to the second improvement.

\subsection{An improved building block}
Recall that $Y(v,w)$ is the union of two subintervals of $[w]$ each of length $v+1$.  For integers $0 \leq u \leq v \leq w$, let $Z(u,v,w)$ be the set obtained from $Y(v,w)$ by replacing each interval of length $v+1$ by a copy of $Y(u,v)$; more explicitly, define
$$Z(u,v,w):=[0,u] \cup [v-u,v] \cup [w-v,w-v+u] \cup [w-u,w]  \subseteq \mathbb{Z}.$$
Fix $n,h \in \mathbb{N}$.  Take natural numbers $u,v_1,\ldots, v_n,w$ satisfying
$$\max\{(h+1)u,w/(h+2)\})\leq v_1<\cdots<v_n\leq \min\{(h+2)u,w/(h+1)\}$$
(as usual with much flexibility), and set
$$B_i:=Z(u,v_i,w)$$
for $1 \leq i \leq n$.  Arguing as in the previous section, we have:
\begin{equation}\label{eq:Z-inequality}
|hB_1|<\cdots <|hB_n|;
\end{equation}
\begin{equation}\label{eq:Z-small-h}
|jB_1|=\cdots=|jB_n|  \quad \text{for all $j< h$};
\end{equation}
\begin{equation}\label{eq:Z-large-h}
jB_1=\cdots=jB_n=[0,jw]  \quad \text{for all $j> h$}.
\end{equation}
The first and third of these properties already appeared in our analysis of the sets $Y$, and the main novelty here is the equality \eqref{eq:Z-small-h} for $j<h$.

\subsection{Assembling the pieces}

We are now in a position to run a simplified version of the argument from \Cref{prop:putting-together-the-pieces}.\footnote{The full complexity of \Cref{prop:putting-together-the-pieces} is still necessary for our argument in the positive-characteristic setting because we do not know of a way to upgrade the sets $X(s,t)$ to sets enjoying the properties of $Z(u,v,w)$.}

\begin{proof}[Proof of \Cref{thm:extension}]
Fix $n,H \in \mathbb{N}$ and tuples $\tau_1, \ldots, \tau_H, \tau_\infty \in \mathbb{N}^n$ as in the statement of \Cref{thm:extension}.  Without loss of generality, we may assume that $\tau_1,\ldots, \tau_H, \tau_\infty \in [1,n]^n$.  For each $1 \leq h \leq H$, take a sequence of sets $B_{h,1}, \ldots, B_{h,n}$ satisfying \eqref{eq:Z-inequality}, \eqref{eq:Z-small-h}, \eqref{eq:Z-large-h}, as constructed in the previous subsection; for notational simplicity, do so in such a way that the parameter $w$ is the same for all of the $h$'s and is sufficiently large relative to $n,H$. 

For $1 \leq k \leq n$, define the set
$$\tilde A_k:=\prod_{h=1}^H B_{h,\tau_h(k)} \subseteq [w]^H \subseteq \mathbb{Z}^H.$$
For each $1 \leq h \leq H$, the quantities
$$|h\tilde A_1|, \ldots, |h\tilde A_n|$$
are in the desired relative order: The properties \eqref{eq:Z-small-h} and \eqref{eq:Z-large-h} ensure that for each $h' \neq h$ the term $|B_{h',\tau_{h'}(k)}|$ is independent of $k$, and \eqref{eq:Z-inequality} dictates the $h$ contribution.

It remains to transfer this construction to the integers and to handle $h>H$.  To this end, define the map $\varphi: \mathbb{Z}^H \to \mathbb{Z}$ via
$$\varphi(x_1,\ldots, x_H):= x_1+(wH+1)x_2+\cdots+(wH+1)^{H-1}x_H.$$
The map $\varphi$ is a Freiman homomorphism and restricts to a bijection $[wH]^H \to [(wH+1)^H-1]$.  In particular, $|h\varphi(A)|=|hA|$ for every subset $A \subseteq [w]^H$ and natural number $h \leq H$.  For $1 \leq k \leq n$, define the set
$$A_k:=\varphi(\tilde A_k)+\{0,(wH+1)^{H}+\tau_\infty(k)\} \subseteq \mathbb{Z}.$$
For each $1 \leq h \leq H$, the sumset $hA_k$ is a disjoint union of $h+1$ copies of $\varphi(h\tilde A_k)$ (disjointness is ensured by $(wH+1)^{H}+\tau_\infty(k)>(wH+1)^H$), and hence the quantities
$$|hA_1|, \ldots, |hA_n|$$
are in the desired relative order.  We now turn to $h>H$.  For all $1 \leq k \leq n$ we have the identity
$$(H+1)\tilde A_k=\varphi((H+1)\tilde A_k)=\varphi([0,w(H+1)]^H)=[0,(1+1/H)((wH+1)^H-1)].$$
Since this interval is longer than all of the shifts $(wH+1)^{H}+\tau_\infty(k)$ (due to $w$ being sufficiently large), the sumset $hA_k$ is a single long interval starting at $0$ whenever $h>H$.  It follows that for such $h$ the order of the quantities $|hA_1|, \ldots, |hA_n|$ is the same as the order of $\max(A_1), \ldots, \max(A_n)$; this order is given by $\tau_\infty$ because $\max(\varphi(\tilde A_k))=(wH+1)^H-1$ is independent of $k$.
\end{proof}

It could be interesting to find an analogous construction in the positive-characteristic setting.

\section{An alternative approach in the integers}\label{sec:old-construction}

In this section we describe an alternative construction which establishes the following special case of \Cref{thm:main}.

\begin{theorem}\label{thm:old}
Let $n,R \in \mathbb{N}$.  There are natural numbers $h_1<\cdots<h_R$ such that the following holds: For any permutations $\sigma_1, \ldots, \sigma_R \in \mathfrak{S}_n$, there exist finite subsets $A_1, \ldots, A_n \subseteq \mathbb{Z}$ such that
$$|h_rA_1|, \ldots, |h_rA_n| \quad \text{has the same relative order as $\sigma_r$}$$
for each $1 \leq r \leq R$.
\end{theorem}

We prove this theorem in the following three subsections, and in the last subsection we compare it with \Cref{prop:zero} and describe why both constructions are of interest.

\subsection{Preliminary lemmas}
The following simple lemma lets us estimate the size of an $h$-fold iterated sumset of a union of sets.  As usual, let $A+B:=\{a+b: a \in A, b \in B\}$ denote the Minkowski sum, and use the convention $0A=\{0\}$ for any $A$.

\begin{lemma}\label{lem:sumset-of-union}
Let $\ell,h$ be natural numbers, and let $A_1, \ldots, A_\ell$ be nonempty finite subsets of an abelian group each containing the identity.  Then the set $A:=A_1 \cup \cdots \cup A_\ell$ satisfies
\begin{enumerate}[(i)]
    \item $hA \subseteq hA_1+\cdots+hA_\ell$ and in particular $|hA| \leq \prod_{i=1}^\ell |hA_i|$;
    \item $h_1A_1+\cdots+h_\ell A_\ell \subseteq hA$ for any nonnegative integers $h_1, \ldots, h_\ell$ summing to at most $h$.
\end{enumerate}
\end{lemma}

\begin{proof}
The lemma follows from the identity
$$hA=\bigcup_{h_1+\cdots+h_\ell=h}\left(h_1 A_1+\cdots+h_\ell A_\ell \right)$$
and the fact that $0A_i \subseteq 1A_i \subseteq 2A_i \subseteq \cdots$ for each $i$ (due to $0 \in A_i$).
\end{proof}

In the sequel, where $A$ is a finite set of integers, we will apply Part (i) together with trivial upper bounds of the form $|hA| \leq 1+h(\max(A)-\min(A))$ (from $hA \subseteq [h\min(A),h\max(A)]$).  We will obtain lower bounds from Part (ii) with $h_1=\cdots=h_\ell=\lfloor h/\ell\rfloor$; it will transpire that the sets $h_i A_i$ are ``additively independent to order $h/\ell$'', in a sense that will be let us (iteratively) apply the following lemma.
%and $\left|h_1A^{(1)}+\cdots+h_\ell A^{(\ell)}\right|=\prod_{i=1}^\ell \left| h_iA^{(i)} \right|$, with the latter quantity easy to compute.

\begin{lemma}\label{lem:sumset-full-expansion}
Let $A,B$ be nonempty finite subsets of an abelian group.  If $(A-A) \cap (B-B)=\{0\}$, then $|A+B|=|A| \cdot |B|$.
\end{lemma}
\begin{proof}
We must show that if $a_1,a_2 \in A$ and $b_1,b_2 \in B$ satisfy $a_1+b_1=a_2+b_2$, then $a_1=a_2$ and $b_1=b_2$.  The hypothesis rearranges to $a_1-a_2=b_2-b_1$; since this quantity lies in $(A-A) \cap (B-B)$, it must vanish.
\end{proof}

We record that, for sets of integers, the hypothesis of this lemma is satisfied if there is some $X \in \mathbb{N}$ such that $\max(A)-\min(A)<X$ and all pairs of elements of $B$ differ by at least $X$.

%\subsection{Nathanson's original question}
%The construction for \Cref{thm:main} is a bit notation-intensive, so we first present the simpler case corresponding to Nathanson's original question.

%For $N \in \mathbb{N}$, write $[N]:=\{0,1,2,\ldots, N\}$ (including $0$ is notationally convenient but a bit non-standard).  For $A \subseteq \mathbb{Z}$ and $\lambda \in \mathbb{N}$, write $\lambda \cdot A:=\{\lambda a: a \in A\}$ for the dilation of $A$ by $\lambda$.

%\begin{proposition}\label{prop:toy-case}
%There are finite subsets $A,B \subseteq \mathbb{Z}$ and natural numbers $h_1<h_2<h_3$ such that $|h_1A|<|h_1B|$, $|h_2B|<|h_2A|$, and $|h_3A|<|h_3B|$.
%\end{proposition}

\subsection{The main estimate}\label{sec:main-estimate}
For $A \subseteq \mathbb{Z}$ and $\lambda \in \mathbb{N}$, write $\lambda \cdot A:=\{\lambda a: a \in A\}$ for the dilation of $A$ by $\lambda$ (not to be confused with the $\lambda$-fold sumset).

Let $0=\alpha_0<\alpha_1<\cdots<\alpha_d$ be nonnegative integers, and let $\gamma \geq 1$ be a natural number.  Assume that any two $\alpha_i$'s differ either by at most $\gamma-1$ or by at least $\gamma+2$.  Define an equivalence relation $\sim$ on $[d]$ by declaring that $$i \sim j \quad \text{if} \quad |\alpha_i-\alpha_j|\leq \gamma-1$$ and then taking the closure.  Each equivalence class is of the form $C=\{i, i+1, \ldots, j\}$ for some $i \leq j$; write $C_{\min}:=\alpha_i$ and $C_{\max}:=\alpha_j$.  For example, if the $\alpha_i$'s are $0,1,7,9,10,20,30,32$ and $\gamma=3$, then the equivalence classes are $\{0,1\}, \{2,3,4\},\{5\},\{6,7\}$, and the equivalence class $C=\{2,3,4\}$ has $C_{\min}=\alpha_2=7$ and $C_{\max}=\alpha_4=10$.  One should think of $\sim$ as splitting the $\alpha_i$'s into clumps with small gaps, where $\gamma$ determines the ``width'' of the allowed gaps.

The following sumset growth estimate is the main ingredient in the proof of \Cref{thm:old}.

\begin{lemma}\label{lem:sumset-at-scales}
Let $0=\alpha_0<\alpha_1<\cdots<\alpha_d$ be nonnegative integers, and let $\gamma \geq 1$ be a natural number.  Assume that any two $\alpha_i$'s differ either by at most $\gamma-1$ or by at least $\gamma+2$.  Define the equivalence relation $\sim$ as above.  For $M \in \mathbb{N}$, define the set
$$A:=\bigcup_{i=0}^d M^{\alpha_i} \cdot [M]$$
and the parameter $h:=M^\gamma$.  For $M$ large, we have that
$$|h A|\asymp_d \prod_{C \in [d]/\sim} M^{C_{\max}-C_{\min}+\gamma+1}.$$
%where the implied constants in the ``$\asymp$'' depend only on $\alpha_1, \ldots, \alpha_d,\gamma$.
\end{lemma}

\begin{proof}
We begin with the upper bound.  For each $C \in [d]/\sim$, define the set
$$\omega(C):=\sum_{i \in C} h(M^{\alpha_i} \cdot [M])$$
(with $\sum$ denoting Minkowski sum).  From the trivial inclusion
$$
\omega(C) \subseteq [d M^{C_{\max}+\gamma+1}]$$
and the fact that every element of $\omega(C)$ is an integer multiple of $M^{C_{\min}}$, we see that
$$|\omega(C)|\ll_d M^{C_{\max}-C_{\min}+\gamma+1}.$$
Applying \Cref{lem:sumset-of-union}(i) and taking a product over all of the equivalence classes, we conclude that
$$|h A|\ll_d \prod_{C \in [d]/\sim} M^{C_{\max}-C_{\min}+\gamma+1},$$
as desired.

We now turn to the lower bound.  Let $h':=\lfloor h/d\rfloor$, and set
$$\omega'(C):=\sum_{i \in C} h'(M^{\alpha_i} \cdot [M]).$$  Recall that every element of $\omega'(C)$ is a multiple of $M^{C_{\min}}$ and in particular distinct elements of $\omega'(C)$ differ by at least $M^{C_{\min}}$.  Recall also the trivial inclusion $\omega'(C) \subseteq [M^{C_{\max}+\gamma+1}]$.  We now use \Cref{lem:sumset-full-expansion} and iterative applications of \Cref{lem:sumset-full-expansion} (see the remark following that lemma, and recall the definition of $\sim$) to conclude that
$$|hA|\geq \prod_{C \in [d]/\sim} |\omega'(C)|.$$
It remains to show that
$$|\omega'(C)| \gg_d M^{C_{\max}-C_{\min}+\gamma+1}$$
for each $C$.  Write $C=\{i,i+1, \ldots, j\}$, with $C_{\min}=\alpha_i$ and $C_{\max}=\alpha_j$.  We expand the sumset representation of $\omega'(C)$ term-by-term.  We start with
$$h' (M^{\alpha_i} \cdot [M])= M^{\alpha_i} \cdot [h'M].$$
The definition of the equivalence relation $\sim$ ensures that $M^{\alpha_i}h'M\asymp_d M^{\alpha_i+1+\gamma}>M^{\alpha_{i+1}}$.  Hence
$$h' (M^{\alpha_i} \cdot [M])+h' (M^{\alpha_{i+1}} \cdot [M])= M^{\alpha_i} \cdot [h'M+h'M^{\alpha_{i+1}-\alpha_i+1}] \supseteq M^{\alpha_i} \cdot [h'M^{\alpha_{i+1}-\alpha_i+1}].$$
Continuing in this fashion, we obtain
$$\omega'(C)\supseteq M^{\alpha_i} \cdot [h'M^{\alpha_{j}-\alpha_i+1}].$$
Unraveling the definitions, we find that the set on the right-hand side has size
$$1+h'M^{C_{\max}-C_{\min}+1}\gg_d M^{C_{\max}-C_{\min}+\gamma+1};$$
this completes the proof.
\end{proof}

\subsection{Proof of \Cref{thm:old}}
Our construction for \Cref{thm:old} will use sets of the form analyzed in \Cref{lem:sumset-at-scales}.  The parameter $M$ will be a large natural number whose exact value is unimportant.  The important point is picking the ``scales'' $\alpha_i$ appropriately so that they can be satisfactorily ``grouped together'' by various values of $\gamma$.  There is substantial flexibility in executing this strategy (especially regarding numerics).  Unfortunately there is also a fair bit of unavoidable notation.

Fix natural numbers $n,R$ and permutations $\sigma_1, \ldots, \sigma_R \in \mathfrak{S}_n$, as in the statement of \Cref{thm:old}.   For each $1 \leq k \leq n$, we will construct an increasing sequence of nonnegative integers
\begin{equation}\label{eq:sequences}
0=\alpha_{k,0}<\alpha_{k,1}<\alpha_{k,2}<\cdots<\alpha_{k,d}
\end{equation}
(for $d$ some constant depending on $R$).  We will also construct a sequence of natural numbers
$$\gamma_1<\cdots<\gamma_R.$$
Our sequences will be compatible in the following sense.  Fix any $1 \leq r \leq R$.  For each $1 \leq k \leq n$, the sequence in \eqref{eq:sequences} will have the property that adjacent elements never differ by $\gamma_r$ or $\gamma_{r}+1$, so we can define an equivalence relation $\sim^{r,k}$ on $[d]$, with width parameter $\gamma_r$, as in the beginning of \Cref{sec:main-estimate}.  Recall that for $C=\{i,i+1,\ldots, j\} \in [d]/\sim^{r,k}$,  we write $C_{\min}=\alpha_{k,i}$ and $C_{\max}=\alpha_{k,j}$.  Consider the quantities
$$E(r,k):=\sum_{C \in [d]/\sim^{r,k}} \left(C_{\max}-C_{\min}+\gamma_r+1\right),$$
which resemble the exponents from \Cref{lem:sumset-at-scales}.  The remaining task is choosing the parameters $\alpha_{k,i},\gamma_r$ so that the $E(r,k)$'s have the desired relative order for each $r$.

\begin{proposition}\label{prop:main-construction}
Let $n,R \in \mathbb{N}$, and let $\sigma_1, \ldots, \sigma_R \in \mathfrak{S}_n$ be permutations.  Then there exist sequences $0=\alpha_{k,0}<\alpha_{k,1}<\cdots<\alpha_{k,d}$ (for $1 \leq k \leq n$) and $\gamma_1<\cdots<\gamma_R$ as above such that
$$E(r,1), \ldots, E(r,n) \quad \text{has the same relative order as $\sigma_r$}$$
for each $1 \leq r \leq R.$
\end{proposition}

\begin{proof}
For each $1 \leq r \leq R$, set $\gamma_r:=(10n)^{10r}$.  We will construct each sequence \eqref{eq:sequences} as a $2 \times \cdots \times 2$ generalized arithmetic progression with rapidly increasing side lengths.  For each $k$, let $\alpha_{k,0}<\alpha_{k,1}<\cdots<\alpha_{k,d}$ (with $d=2^R-1$) be the elements of the set
$$\sum_{s=1}^R \left\{0,\sigma_s(k) \cdot \frac{\gamma_s}{10n}\right\};$$
it is clear that consecutive elements of this sequence \eqref{eq:sequences} do not differ by $\gamma_r$ or $\gamma_r+1$.  Let us calculate $E(r,k)$.  The equivalence relation $\sim^{r,k}$ has $2^{R-r}$ equivalence classes $C$, each satisfying $$C_{\max}-C_{\min}=\frac{1}{10n}\sum_{s=1}^r \sigma_s(k) \gamma_s.$$
Thus we have
$$E(r,k)=2^{R-r}(\gamma_r+1)+
\frac{2^{R-r}}{10n}\sum_{s=1}^r \sigma_s(k) \gamma_s.$$
The $s=r$ term dominates the sum, so the expressions $E(k,r)$ have the desired relative orders.
\end{proof}

We can now deduce \Cref{thm:old}.

\begin{proof}[Proof of \Cref{thm:old}]
Take the parameters $d,\alpha_{k,i},\gamma_r$ as in \Cref{prop:main-construction}.  Let $M \in \mathbb{N}$ be sufficiently large (depending on $d$).  For each $1 \leq k \leq n$, define the set
$$A_k:=\bigcup_{i=0}^d M^{\alpha_{k,i}} \cdot [M],$$
and set
$$h_r:=M^{\gamma_r}$$
for each $1 \leq r \leq R$.  \Cref{lem:sumset-at-scales} tells us that each
$$|h_rA_k| \asymp_d M^{E(r,k)},$$
and \Cref{prop:main-construction} ensures that these quantities have the desired relative order for each $r$.  Notice that the sequence $h_1<\cdots <h_R$ can be taken to depend on only $n,R$ (and in particular to be independent of $\sigma_1, \ldots, \sigma_R$).
\end{proof}

\subsection{Comparison with \Cref{prop:zero}}

The constructions for \Cref{thm:old} and \Cref{prop:zero} both use unions of arithmetic progressions at different scales.  The mechanisms underlying these two construction are quite different, however.

In the construction for \Cref{thm:old}, as one takes higher-order iterated sumsets, the arithmetic progressions ``merge'' in pairs, then in quadruples, and so on, until a sufficiently high-order iterated sumset consists of a single long interval.   This is a fundamentally ``$1$-dimensional'' phenomenon.  Each merging corresponds to a (relative) slow-down in the growth rate of the iterated sumsets; the exact timing of these mergings causes the desired fluctuations in the relative sizes of iterated sumsets.  This approach requires some (rough) quantitative estimates on the sizes of iterated sumsets to ensure that the main fluctuations are larger than accumulated error terms; it is for the sake of this balancing act that the sequence of $h$'s grows very quickly.

In the construction for \Cref{prop:zero}, by contrast, each arithmetic progression is involved in only one merging (as described in \Cref{sec:zero}).  Each merging again causes a fluctuation in the relative sizes of the iterated sumsets, and the product structure of the construction (embedded in $\mathbb{Z}$ by means of different scales) ensures that these different pieces remain completely independent.  This is a fundamentally ``high-dimensional'' phenomenon.  Independence makes the analysis correspondingly ``softer'' in the sense of not requiring \emph{quantitative} estimates on the relative sizes of iterated sumsets; it is for this reason that we may prescribe the sequence of $h$'s.

The advantage of the latter approach is (obviously) that its greater flexibility allows us to prove stronger results.  The principle of the construction is sufficiently general that it also works in the positive-characteristic setting with only minor modifications (\Cref{prop:positive}).  The interest of the former approach is that it directly harnesses the diversity of scales available in the integers, rather than ``cheating'' by embedding a higher-dimensional object.  We are optimistic that these ideas will find applications in other problems.

\section*{Acknowldgements}
The author was supported in part by the NSF Graduate Research Fellowship Program under
grant DGE–203965.  I thank Noga Alon and Jacob Fox for helpful comments.

\end{document}